\documentclass[12pt,reqno]{amsart}
\usepackage{hyperref}
\usepackage{amsmath}
\usepackage{amssymb,amsthm}
\usepackage{amsfonts,cmtiup,comment,stmaryrd}
\usepackage{cite}
\usepackage{mathrsfs,cmtiup}

\makeatletter
\def\blfootnote{\xdef\@thefnmark{}\@footnotetext}
\makeatother

\newtheorem*{theorema}{Theorem A}
\newtheorem*{corollaryb}{Corollary B}
\newtheorem{theorem}{Theorem}[section]
\newtheorem{lemma}[theorem]{Lemma}
\newtheorem{proposition}[theorem]{Proposition}

\theoremstyle{definition}

\newtheorem*{definition*}{Definition}
\numberwithin{equation}{section}

\begin{document}

\title{Length parameters of finite groups\\ and their Hall subgroups}

\author{Evgeny Khukhro}
\address{E. I. Khukhro: Charlotte Scott Research Centre for Algebra, University of Lincoln, U.K.}
\email{khukhro@yahoo.co.uk}

\author{Pavel Shumyatsky}

\address{P. Shumyatsky: Department of Mathematics, University of Brasilia, DF~70910-900, Brazil}
\email{pavel@unb.br}

\thanks{The first author was partially supported by the International Center for Mathematics at SUSTech in Shenzhen. The second author was supported by FAPDF and CNPq.}

\begin{abstract} Let $p$ be an odd prime, and $\pi$ a set of primes such that $2\in \pi$ and $p\in \pi$. It is proved that if a finite group $G$ has a Hall $\pi$-subgroup $H$, then the non-$p$-soluble length of $G$ is bounded above by the generalized Fitting height of $H$. The proof uses the fact, obtained in \cite{bel} using the classification of finite simple groups, that a finite simple group of order divisible by $p$ cannot have a nilpotent Hall $\{2,p\}$-subgroup. As a corollary, it is proved that if in addition $H$ is soluble, then the non-$p$-soluble length of $G$ is bounded above  by $2l_2(H)+1$, where $l_2(H)$ is the $2$-length of $H$.

\end{abstract}

\maketitle

\section{Introduction}

All groups considered in this paper are finite.
Let $\pi$ be a set of primes. Recall that a Hall $\pi$-subgroup of a finite group $G$ is a subgroup $H$ such that its order $|H|$ is divisible only by the primes in $\pi$, while its index $|G:H|$ is not divisible by any prime in $\pi$. Every soluble finite group, for any set of primes $\pi$, has Hall $\pi$-subgroups,  which satisfy  analogues of Sylow theorems by the theorems of  Hall and \v{C}unihin; see for example, \cite[Main Theorem VI.1.7]{hup1}. Unlike Sylow $p$-subgroups, Hall $\pi$-subgroups may not satisfy such theorems, or even not exist, in non-soluble finite groups for some sets of primes $\pi$, and the study of Hall subgroups is an important area of modern finite  group theory.

In \cite{bel,mor}, there have been obtained some criteria for the existence of nilpotent or
abelian Hall $\pi$-subgroups in finite groups. In particular, with the use of the classification of finite simple groups, it was proved in \cite[Theorem~2.2]{bel} that for any odd prime $p$, a finite simple group of order divisible by $p$ cannot have a nilpotent Hall $\{2,p\}$-subgroup. In this note we use this result to show that if, for an odd prime $p$ and a set of primes $\pi$ such that $2\in \pi$ and $p\in \pi$, a finite group $G$ has a Hall $\pi$-subgroup $H$, then the non-$p$-soluble length of $G$ is bounded above by the generalized Fitting height of $H$.

 Both the non-$p$-soluble length and the generalized Fitting height are important parameters of non-soluble finite groups, and bounds for these parameters enable inductive arguments for finite and profinite groups, as illustrated by a few examples at the end of this Introduction. But first we recall the definitions and notation in order to state our results. The generalized Fitting subgroup $F^*(G)$ of a finite group $G$ is the characteristic subgroup generated by the Fitting subgroup $F(G)$ and all the subnormal quasisimple subgroups (a~group $Q$ is quasisimple if $[Q,Q]=Q$ and $Q/Z(Q)$ is simple); see, for example, \cite[\S\,X.13]{bla-hup3}. The \textit{generalized Fitting series} of $G$ is defined recursively by $F_{1}^{*}(G)=F^*(G)$ and $F_{i+1}^{*}(G)/F_{i}^{*}(G)=F^{*}(G/F_{i}^{*}(G))$; see~\cite{khu-shu15a, khu-shu15b}. The least number $h$ such that $F_{h}^{*}(G)=G$ is called the
\textit{generalized Fitting height} $h^{*}(G)$ of $G$.

For a given prime $p$, the \textit{non-$p$-soluble length} $\lambda_p(G)$ of a finite group $G$ is the minimum number of non-$p$-soluble factors in a normal series where each factor is either $p$-soluble or a direct product of non-abelian simple groups of orders divisible by $p$; see~\cite{khu-shu15}. In particular, $G$ is $p$-soluble exactly when its non-$p$-soluble length is~$0$.

We can now state our main result.

\begin{theorema}\label{t}
Let $p$ be an odd prime, and $\pi$  a set of primes such that $2\in \pi$ and $p\in \pi$.
Suppose that a finite group $G$ has a Hall $\pi$-subgroup $H$. Then the non-$p$-soluble length of $G$ is bounded above by the generalized Fitting height of $H$, that is, $\lambda_p(G)\leqslant h^*(H)$.
\end{theorema}

The proof relies on the classification of finite simple groups. As a corollary, we obtain a stronger conclusion in the case of a soluble Hall subgroup. Recall that the $p$-length $l_p(R)$ of a $p$-soluble finite group $R$ is the number of $p$-factors in a shortest normal series of $R$ in which each factor is either a $p$-groups or has order coprime to~$p$; see~\cite{ha-hi}.

\begin{corollaryb}
 \label{c}
Let $p$ be an odd prime, and $\pi$  a set of primes such that $2\in \pi$ and $p\in \pi$.
 Suppose that a finite group $G$ has a soluble Hall $\pi$-subgroup $H$. Then the non-$p$-soluble length of $G$ is bounded above by $2l_2(H)+1$, where $l_2(H)$ is the $2$-length of $H$.
\end{corollaryb}

This corollary is of independent interest, since the Fitting height of a soluble finite group cannot be bounded above in terms of its $2$-length.

We introduced the non-$p$-soluble length of a finite group in \cite{khu-shu15}, and the generalized Fitting height in \cite{khu-shu15a,khu-shu15b}. But the bounds for the non-soluble length (which is the same as non-$2$-soluble length) and generalized Fitting height had been implicitly used earlier, for example, in the reduction of the Restricted Burnside Problem to soluble and nilpotent groups in the Hall--Higman paper \cite{ha-hi}. In several recent papers  bounds for the non-soluble and non-$p$-soluble length of finite groups were studied inter alia in relation to Engel-type conditions and length of group laws, as, for example, in \cite{ flp, gur-tra, khu-shu17,khu-shu21}.

Such bounds also play an important role
in the study of profinite groups. For example, Wilson~\cite{wil} obtained a bound for the non-soluble length of finite quotients of a torsion profinite group, thus reducing the problem of its local finiteness to the case of pro-$p$-groups (the latter was then settled by Zelmanov~\cite{zel}).
In some  recent papers on probabilistic properties of profinite groups, bounds for non-soluble length of finite quotients provide reduction to the cases of prosoluble groups and Cartesian products of finite simple groups; see, for example,~\cite{det-mor-shu}.
Similar type of reduction is used in the study of the conciseness property of words in profinite groups, as for example, in \cite{det-mor-shu23, khu-shu23}, as well as in the study of profinite groups with restrictions on centralizers, as in \cite{shu20,shu20a,shu}.

 \section{Preliminaries}
 Let $\pi$ be a set of primes. Recall that $H$ is a Hall $\pi$-subgroup of a finite group $G$ if its order $|H|$ is divisible only by the primes in $\pi$, while its index $|G:H|$ is not divisible by any prime in $\pi$.
It follows from Lagrange's theorem that if $N$ is a normal subgroup of $G$, then $H\cap N$ is a Hall $\pi$-subgroup of $N$. This immediately implies that if $S$ is a subnormal subgroup of~$G$, then $H\cap S$ is a Hall $\pi$-subgroup of $S$. It is also clear that $HN/N$ is a Hall $\pi$-subgroup of the quotient $G/N$.

 Let $G$ be a finite group, and $p$ a prime. Recall that the \textit{non-$p$-soluble length} $\lambda_p(G)$ of $G$ is the minimum number of non-$p$-soluble factors in a normal series where each factor is either $p$-soluble or a direct product of non-abelian simple groups of orders divisible by $p$. In particular, $G$ is $p$-soluble exactly when its non-$p$-soluble length is~$0$.
It follows from the definition that for a normal subgroup $N$ the non-$p$-soluble length satisfies $\lambda_p(N)\leqslant   \lambda_p(G)$ and $\lambda_p(G/N)\leqslant  \lambda_p(G)$.
 It is also clear that an extension of a normal subgroup of non-$p$-soluble length $k$ by a group of non-$p$-soluble length $l$ has non-$p$-soluble length at most $k+l$.

One of the ways of realizing the non-$p$-soluble length of $G$ is based on the so-called $p$-kernel subgroups. Let $R_p(G)$ denote the $p$-soluble radical of $G$, that is, the largest normal $p$-soluble subgroup. If $R_p(G)\ne G$, that is, $G$ is not $p$-soluble, then in the quotient $\bar G=G/R_p(G)$ the socle $\mathop{Soc}(\bar G)$, which is the product of all minimal normal subgroups of $\bar G$, is a direct product $\mathop{Soc}(\bar G)=S_1\times\dots\times S_m$ of non-abelian simple groups $S_i$ of orders divisible by $p$.
The group $G$ induces by conjugation a permutational action on the set $\{S_1, \dots , S_m\}$. We call the kernel $K_p(G)$ of this action the \emph{$p$-kernel subgroup} of $G$.
Clearly, $K_p(G)$ is the full inverse image in $G$ of $\bigcap_i N_{\bar G}(S_i)$; see \cite{khu-shu15}.

The groups of outer automorphisms of finite simple groups are soluble, which is known as the validity of Schreier's Conjecture confirmed by the classification of finite simple groups; see~\cite{gls}. It follows that the quotient of $K_p(G)$ by the inverse image of the socle $\mathop{Soc}(\bar G)$ is soluble; in particular, $K_p(G)$ has non-$p$-soluble length~1; see \cite{khu-shu15}.

For uniformity, we also put $K_p(G)=G$ if $R_p(G)= G$, that is, when $G$ is $p$-soluble and its non-$p$-soluble length is~$0$. Thus, we have in all cases the following.

\begin{lemma}[{see, for example, \cite[Lemma~2.1]{khu-shu15}}]
\label{l-1}
The non-$p$-soluble length of $K_p(G)$ is at most~$1$.
\end{lemma}

One can also define the higher $p$-kernel subgroups by induction: $K_{p,1}(G)= K_{p}(G)$, and $K_{p,i+1}(G)$ is the full inverse image of $K_{p}(G/K_{p,i}(G))$. As shown in \cite{khu-shu15}, the non-$p$-soluble length of $G$ is the least $i$ such that $G/K_{p,i}$ is $p$-soluble, although we do need this fact.

Recall that the $p$-length $l_p(G)$ of a $p$-soluble finite group $G$ is the number of $p$-factors in a shortest normal series of $G$ in which each factor is either a $p$-groups or has order coprime to~$p$; see~\cite{ha-hi}. By the Feit--Thompson theorem \cite{fei-tho}, $2$-soluble groups are soluble. A soluble finite group has Hall $\pi$-subgroups for any set of primes $\pi$; see, for example, \cite[Main Theorem VI.1.7]{hup1}.

\section{Generalized Fitting subgroup of a Hall subgroup}

The proof of the main Theorem~A will follow by a straightforward induction from the results of this section relating the Fitting subgroup and the generalized Fitting subgroup of a Hall subgroup with the $p$-kernel subgroup  of the group.

 \begin{lemma}\label{l-2}
 Let $p$ be an odd prime, and $\pi$  a set of primes such that $2\in \pi$ and $p\in \pi$. Suppose that a finite group $G$ has a Hall $\pi$-subgroup $H$. Then the Fitting subgroup of $H$ is contained in the $p$-kernel subgroup of $G$, that is, $F(H)\leqslant K_p(G)$.
 \end{lemma}

\begin{proof}
We can clearly assume that the $p$-soluble radical of $G$ is trivial. Let $\mathop{Soc}(G)=S_1\times\cdots \times S_k$ be the socle of $G$, where the $S_i$ are nonabelian simple groups of orders divisible by~$p$. Of course, the orders of $S_i$ are also divisible by $2$ by the Feit--Thompson theorem~\cite{fei-tho}.
We argue by contradiction: suppose that $g\in F(H)\setminus K_p(G)$, which means that $S_i^g=S_j$ for some $i\ne j$. Consider $S_i\cap H$, which is a Hall $\pi$-subgroup of $S_i$, since $S_i$ is subnormal in $G$. Note that both $2$ and $p$ divide $|S_i\cap H|$, since $2$ and $p$ divide $|S_i|$ and $2,p\in \pi$.

The commutator subgroup $[S_i\cap H ,g]$ is generated by the commutators $[a,g]=a^{-1}a^g$ for $a\in S_i\cap H$. Since $a^g\in S_j\ne S_i$, the projection of $[S_i\cap H ,g]$ onto $S_i$ is equal to $S_i\cap H$. Since $g$ belongs to $F(H)$, which is normalized by $S_i\cap H$, we have $[S_i\cap H,g]\leqslant F(H)$, so that the subgroup $[S_i\cap H,g]$ is nilpotent. Hence its projection onto $S_i$ is also nilpotent. As a result, $S_i$ has a nilpotent Hall $\pi $-subgroup $S_i\cap H$, and therefore also a nilpotent Hall $\{2,p\}$-subgroup of order divisible by $p$. This contradicts \cite[Theorem~2.2]{bel}.
\end{proof}

\begin{proposition}\label{p}
 Let $p$ be an odd prime, and $\pi$  a set of primes such that $2\in \pi$ and $p\in \pi$. Suppose that a finite group $G$ has a Hall $\pi$-subgroup $H$. Then the generalized Fitting subgroup of $H$ is contained in the $p$-kernel subgroup of $G$, that is, $F^*(H)\leqslant K_p(G)$.
\end{proposition}

\begin{proof}
 We can clearly assume that the $p$-soluble radical of $G$ is trivial. Let $\mathop{Soc}(G)=S_1\times\cdots \times S_k$ be the socle of $G$, where the $S_i$ are nonabelian simple groups of orders divisible by~$p$. The orders of $S_i$ are also divisible by $2$ by the Feit--Thompson theorem~\cite{fei-tho}. If $F^*(H)=F(H)$, we are done by Lemma~\ref{l-2}. Otherwise, and assuming the opposite, $K_p(G)\cap F^*(H)$ is a proper normal subgroup of $F^*(H)$ containing $F(H)$ by Lemma~\ref{l-2}. Since $F^*(H)/F(H)$ is a direct product of non-abelian simple factors, at least one of these factors, say, $U$ is not in the image of $K_p(G)\cap F^*(H)$ in $F^*(H)/F(H)$. Therefore $U$ has trivial intersection with the image of $K_p(G)\cap F^*(H)$ in $F^*(H)/F(H)$ and centralizes this image. Let $\hat U $ be the inverse image of $U$ in $H$. Note that $\hat U\cap K_p(G)\leqslant F(H)$. Then
\begin{equation}\label{e}
 [[K_p(G)\cap H, \hat U],\hat U]\leqslant [K_p(G)\cap F^*(H), \hat U]\leqslant F(H).
\end{equation}

We now choose an element $g\in \hat U$ such that its image in $U$ has prime order $q\geqslant 3$. Then $g\not\in K_p(G)$ and $g$ has an orbit of length $q$ in the permutational action on the set $\{S_1,\dots ,S_k\}$. Renumbering the factors $S_i$ if necessary, we can assume without loss of generality that $\{S_1,\dots, S_q\}$ is such an orbit where $S_1^g=S_2\ne S_1$ and $S_2^g=S_3\ne S_1$. The commutator subgroup $ [S_1\cap H ,g]$ is contained in $S_1\times S_2$ and its projection onto $S_1$ is equal to $S_1\cap H$, since it is generated by the commutators $a^{-1}a^g$, where $a$ runs over all elements of $S_1\cap H$, while $a^g\in S_2$.

The projection of the subgroup $[[S_1\cap H ,g],g]$ onto $S_1$ is also equal to $S_1\cap H$. Indeed, it is generated by the commutators of the form $(ab)^{-1}(ab)^g$, where $ab$ runs over all elements of $[S_1\cap H ,g]$ with $a\in S_1$ and $b\in S_2$. Here, $a$ runs over all elements of $S_1\cap H$ by the above, while $(ab)^g\in S_2\times S_3$. Thus we obtain that $[[S_1\cap H ,g],g]$ is generated by elements of the form $(ab)^{-1}(ab)^g=a^{-1}\cdot b^{-1}a^gb^g$, where $a$ runs over all elements of $S_1\cap H$, while $b^{-1}a^gb^g\in S_2\times S_3$.

By \eqref{e} we have $[[S_1\cap H ,g],g] \leqslant F(H)$, so that $[[S_1\cap H ,g],g]$ is nilpotent. Hence its projection onto $S_1$, which is $S_1\cap H$, is also nilpotent. But $S_1\cap H$ is a Hall $\pi$-subgroup of $S_1$, since $S_1$ is subnormal. Thus, $S_1$ has a nilpotent Hall $\pi$-subgroup, and therefore also a nilpotent Hall $\{2,p\}$-subgroup of order divisible by $p$, since both $2$ and $p$ divide $|S_1|$ and $2,p\in \pi$. This contradicts \cite[Theorem~2.2]{bel}.
\end{proof}

\section{Proofs of the main results}

We are now ready to prove Theorem~A.

\begin{proof}[Proof of Theorem~A]
Recall that $p$ is an odd prime, $\pi$ is a set of primes such that $2\in \pi$ and $p\in \pi$, and $G$ is
a finite group that has a Hall $\pi$-subgroup $H$.  We need to show that the non-$p$-soluble length of $G$ is at most the generalized Fitting height of $H$, that is, $\lambda_p(G)\leqslant h^*(H)$. We proceed by induction on $h^*(H)$. If $h^*(H)=1$, then $H=F^*(H)$ is contained in the $p$-kernel subgroup $K_p(G)$ by Proposition~\ref{p}. Then $|G/K_p(G)|$ is not divisible by $p$ and therefore $G/K_p(G)$ is $p$-soluble. Since $\lambda_p(K_p(G))\leqslant 1$ by Lemma~\ref{l-1}, it follows that $\lambda_p(G)\leqslant 1$, as required.

 In general, since $F^*(H)\leqslant K_p(G)$ by Proposition~\ref{p}, the generalized Fitting height of the image $\bar H$ of $H$ in $G/K_p(G)$ is smaller than $h^*(H)$. By induction, $\lambda_p(G/K_p(G))\leqslant h^*(\bar H)\leqslant h^*(H)-1$. Since $\lambda_p(K_p(G))\leqslant 1$ by Lemma~\ref{l-1}, the result follows: $\lambda_p(G)\leqslant 1+\lambda_p(G/K_p(G))\leqslant 1+ h^*(H)-1=h^*(H)$.
\end{proof}

We now prove Corollary~B.

\begin{proof}[Proof of Corollary~B]
Recall that $p$ is an odd prime, $\pi$  is a set of primes such that $2\in \pi$ and $p\in \pi$, and $G$ is
a finite group that has a soluble Hall $\pi$-subgroup $H$.
We want to obtain a bound for the non-$p$-soluble length of $G$ in terms of the $2$-length of $H$. Since $H$ is soluble, it has a Hall $\{2,p\}$-subgroup $T$, which is also a Hall $\{2,p\}$-subgroup of $G$. Clearly, $T$ has a normal series of length at most $2l_2(T)+1$ each factor in which is either a $2$-group or a $p$-group. All these factors are nilpotent, and therefore the Fitting height $h(T)=h^*(T)$ of $T$ is at most $2l_2(T)+1$. The $2$-length of $T$ is at most the $2$-length of $H$. Applying Theorem~A to $G$ and its Hall $\{2,p\}$-subgroup $T$ we obtain $\lambda_p(G)\leqslant h^*(T)\leqslant 2l_2(T)+1\leqslant 2l_2(H)+1$.
\end{proof}

\end{document}